\documentclass{amsart}
\usepackage[utf8]{inputenc}

\usepackage{graphics}
\usepackage{thmtools}
\usepackage{wasysym}
\usepackage[T1]{fontenc}    

\usepackage{amsthm}
\usepackage{amsbsy,amsmath,amssymb,amscd,amsfonts,float}
\usepackage[pagebackref=true]{hyperref}
\usepackage[nameinlink,capitalize,noabbrev]{cleveref}

\usepackage{graphicx,float,latexsym,color}
\usepackage[font={scriptsize,it}]{caption}
\usepackage{subcaption}

\usepackage{makecell}

\usepackage[dvipsnames]{xcolor}

\newtheorem*{theorem*}{Theorem}

\newtheorem{proposition}{Proposition}
\newtheorem{conjecture}{Conjecture}

\theoremstyle{remark}

\theoremstyle{definition}

\hypersetup{
    pdftoolbar=true,        
    pdfmenubar=true,        
    pdffitwindow=false,     
    pdfstartview={FitH},    
    colorlinks=true,       
    linkcolor=OliveGreen,          
    citecolor=blue,        
    filecolor=black,      
    urlcolor=red           
}

\usepackage{lineno}

\usepackage{comment}
\usepackage{microtype}
\usepackage{footnote}

\newcommand{\E}{\mathcal{E}}

\title{Fregier ellipses}
\author{Dominique Laurain}
\date{June, 2022}

\begin{document}

\maketitle

\begin{abstract}
We introduce Fregier ellipses as generalization of the Fregier point and exhibit the amazing angles and areas invariants.
\end{abstract}

\section{Introduction}
Given an ellipse $\E$, the Frégier theorem states that given a point $M$ on the ellipse, all the $M$-rectangular triangles $MNL$ inscribed in the ellipse, have edges $NL$ intersecting at a single point, the ``M-Frégier point``. This paper is about a more general setup where internal angle at $M$ vertex is $\theta$ between $0$ and $\pi$ and we prove some properties observed by Dan Reznik and displayed in videos \cite{reznik2021_fregierI,reznik2021_fregierIII} about envelope $\E'$ of $MN$ chords.

\begin{figure}[H]
    \centering
    \includegraphics[width=.5\textwidth]{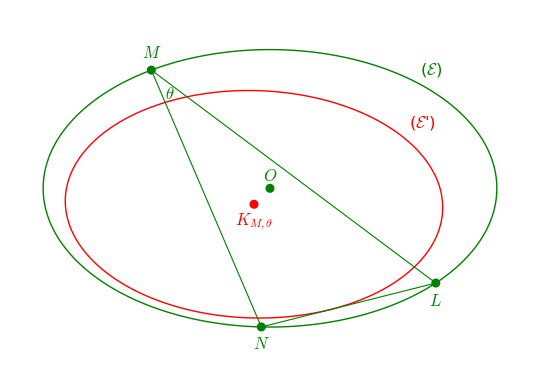}
    \caption{A Frégier ellipse (red color) defined by a point M and an angle $\theta$ }
    \label{fig:fregier_ellipse}
\end{figure}

It should be noticed that the reverse problem, given an angle $\theta$ and a chord $NL$ of the ellipse $\E$ find $M$, has at most two solutions. Get them using cyclic quadrilateral opposite angles characterization : set a triangle $NPL$ on $NL$ with angle $\pi - \theta$ at $P$, and intersect $NPL$ circumcircle with the ellipse.

\subsection*{Main results}  the envelope of $NL$ edges is an ellipse $\E' = \E'_{M,\theta}$. The ellipse is degenerated to a point (M-Frégier point) for $\theta = \frac{\pi}{2}$.
The locus of $\E'$ center is a $\E$-concentric ellipse and, amazingly, $\E'$ has constant area with varying axes lengths, when $M$ is moving around $\E$.
Complementary ``optic`` property : angle of the two tangents from $M$ to $\E'$ as constant $\pi - 2\theta$ value. For special value,  $\theta = \pi/3$ or $\theta = 2\pi/3$, $MNL$ is circumscribing $\E'$.

\subsection*{Subsidiary result}  when $\E$ is a circle then $\E'$ is a concentric circle (a ``Fregier circle``), and given a Poncelet circle-ellipse configuration, the sum of areas of $P_i$-Fregier circles is the same for every $P_1 \dots P_n$ orbit. We prove it for $n = 3$ and conjecture it for $n > 3$.

\section{Computing the envelope}
\label{sec:envelope}
Algebraic computations in proofs are done using trilinear coordinates with $ABC$ reference triangle being the isosceles billiard 3-orbit with $A$ on ellipse minor axis.
If $a$ and $b$ are semi-axis lengths then $AB = AC = \frac{2s}{3 + h}$ and $BC = \frac{2(1 + h)s}{3 + h}$ where $s$ is $ABC$ semi-perimeter and $h$ is positive root of $(\frac{a}{b})^2 = \frac{(1 + h)(3 - h)}{(1 - h)(3 + h)}$.

\begin{proposition}
Given a point $M$ on an ellipse $\E$, the envelope of chords $NL$ producing a constant angle $\theta$ for inscribed triangles $MNL$ is included in an ellipse $\E'$.
\end{proposition}

\begin{proof}
Assuming the envelope is a conic, five tangents are computed as repetition of construction given by
 \cref{fig:computations} : select one side (picture : $AC$) of the reference triangle $ABC$ and erect an isosceles triangle (picture : $AEC$), then draw chords $MN$ and $ML$ as parallel segments from $M$ with lines $AE$ and $AC$.

 Solving system of five equations for tangents provides an equation, and by duality, equation of $\E'$, envelope for the tangents, is obtained.
 
\begin{figure}
    \centering
    \includegraphics[width=\textwidth]{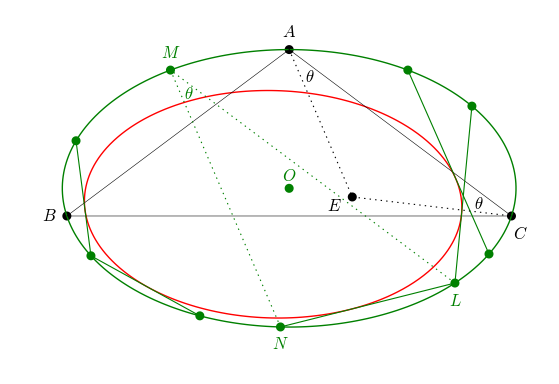}
    \caption{Geometric construction of five tangents of the envelope}
    \label{fig:computations}
\end{figure}

Given real parameter $u$ we set $M = u + 1 : u ( u + 1) : -u$.

Assuming one generic tangent line  $1 : m : n$ is intersecting ellipse $\E$ at two points $M_\pm = -2 m n : (m - n + \pm \sqrt{(m - n)^2 -2(m  + n) + 1} ) n :  -(m - n + 1 \pm \sqrt{(m - n)^2 -2(m  + n) + 1}) m$.

Proof of proposition is done using CAS software : we compute algebraic value of difference between squared cosine of angle $M_1MM_2$ (using cosines law and $M,M_1,M_2$ coordinates) and $\cos^2{\theta}$. Equation in $u$ is deduced when that difference is zero.
Equation is identical to the one got by duality, proving the proposition.
\end{proof}

In the general case, tangent points $T_1$ and $T_2$ are defining the polar line of $M$ with respect to $\E'$. Angle $T_1MT_2$ doesn't depend from $M$ as given in the following proposition.

\begin{proposition}
For $T_1T_2$ the $\E'$-polar line of $M$ , angle $T_1MT_2$ is constant with value $\pi - 2\theta$.
\end{proposition}

\begin{proof}
Line $1 : m :n$ is going through $M$ when $n = \frac{(m u + 1)(u + 1)}{u}$.

The two tangent lines to $\E'$ through $M$ have trilinear coordinates : $(k_1 \pm k_2 u w)u : k_3 \pm k_2 w : (k_4 \pm k_2 (u + 1) w) (u + 1)$ with
$$k_1 = (h^2 u - h u^2 + h^2 + 2 h u + u^2 + 2h + 1)(u + 2)$$
$$k_2 = h u + u^2 + h + u + 1$$
$$k_3 = 2 h u^3 - h^2 u + 3 h u^2 - 2 u^3 - h^2 - 3 u^2 - 2h - 3u - 1$$
$$k_4 = h^2 u^2 + h u^3 + 2h^2 u + 3h u^2 - u^3 + h^2 + 6h u + 2h + 1$$
$$w = (\sqrt{3 + h}) (\sqrt{1 - h}) \cos{\theta}$$

We deduce trilinear coordinates of $T_1$ and $T_2$ by intersecting the two tangent lines with $\E$ and the result in the proposition follows from computation of $\cos^2{ (T_1MT_2)}$ by cosines law. 

\end{proof}

\section{Locus of center and area}
\label{sec:locus}

Refer to \cref{fig:locus} for yellow colored locus center $K$ (depending on $M$ and $\theta$) of ellipse $\E'$.

\begin{proposition}
For a varying $M$ point the locus of $K$ center of $\E'$ is an ellipse concentric with ellipse $\E$.
\end{proposition}

\begin{proof}
Trilinear coordinates $\alpha : \beta : \gamma$ of center of ellipse $K$ are retrieved directly by duality and from the equation of $\E'$  :
$$\alpha = (-(h u + h + 2)h - h(u + 2)u + u^2 + u + 1)(3 - h^2) + (hu + u^2 + h + u + 1)(1 - h)w^2 $$
$$\beta = ( (h u + h + 2) h + h(u + 4)u + u^2 + u + 1)(3 - h^2) + (h*u + u^2 + h + u + 1)(1 + h)w^2 $$
$$\gamma = ( (hu + h + 2)h - hu^2       + u^2 + u + 1)(3 - h^2) + (h u + u^2 + h + u + 1)(1 + h)w^2] $$

Plugging five values ($0$, $\pm 1$, $\pm 2$) for $u$, we get five points $K_1, \dots, K_5$. From them equation of conic locus of $K$ follows :   
$$k_1 \alpha^2 + k_2 \beta^2 + k_3 \gamma^2 + k_4 \beta \gamma + k_5 \gamma \alpha + k_6 \alpha \beta = 0$$
with
$$k_2 = k_3 = ((1 + h)(3 - h)w^2 + (3 - h^2)(3 + h)(1 - h))(w^2 + 3 - h^2)$$
$$k_1 = k_2 (1 + h)^2 $$
$$k_4 = -((1 + h)(3 - h)w^2 + 2(3 - h^2)^2)(1 + 2h - h^2)w^2 - (h^4 - 4h^2 + 4h + 3)(3 - h^2)^2 $$
$$k_5 = k_6 = -((1 + h)(3 - h)w^2 + 2(3 - h^2)^2)(1 - h^2)w^2 - (h^4 + 2h^3 - 2h + 3)(3 - h^2)^2 $$

Because $k_1$,$k_2$,$k_3$ have same signs, the conic is an ellipse and its center is $O = 1 - h : 1 + h : 1 + h$, center of ellipse $\E$.  
\end{proof}

\begin{figure}
    \centering
    \includegraphics[width=\textwidth]{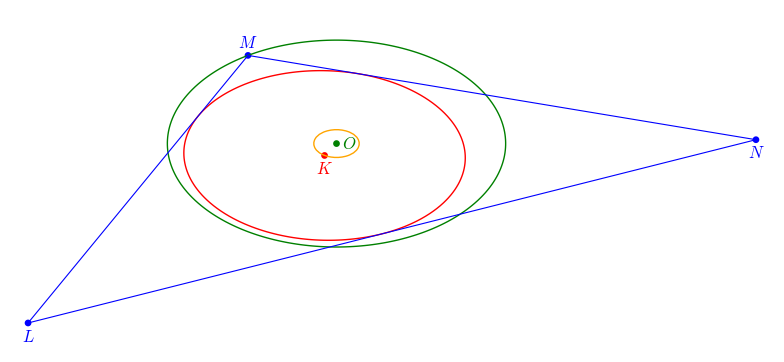}
    \caption{Locus center of envelope}
    \label{fig:locus}
\end{figure}

\begin{proposition}
Squared area of $\E'$ is $k^2 $ times squared area of $\E$ where :
$$k^2 = \frac{\rho^3(\rho+4)^3 \cos^2{\theta}}{((\rho+1)^2 + (2\rho-1)\cos^2{\theta})^3}$$
with $\rho=(1-h^2)/2$.
\end{proposition}

\begin{proof}
Extend $LN$ segment shown on \cref{fig:fregier_ellipse} and intersect with the two $E'$-tangent lines through $M$ to get a triangle $MNL$ inscribing the ellipse $\E'$. The area is computed as an $MNL$ inellipse area by formula given in \cite{mw_inellipse}.

The area of $\E'$ doesn't depend from $u$ (parameter defining $M$) hence the area is constant when $M$ is moving around the ellipse.

\end{proof}

When the outer ellipse $\E$ is a circle, $h=0$ and $\rho=1/2$ giving $k^2 = \cos^2{\theta}$.

\section{Special theta values}
\label{sec:special}

One special angle value is $\theta=\pi/2$ where ellipse $\E'$is reduced to $K$ the M-Fregier point.

Another special angle value is  $\theta=\pi/4$, because the ellipse $\E'$ is ``viewed`` from $M$ point on $\E$ ellipse with a $\pi/2$ angle, $M$ is on the $\E'$  orthoptic circle of the caustic. In that case, $\E'$ semi-axis lengths ($a_1,b_1$) can be easily deduced from the formula for $\E'$ squared area ($\pi^2a_1^2b_1^2)$ and $KM^2 = a_1^2+b_1^2$ .
For generic angle $\theta$, use ``roulette`` construction of orthoptic circle point $Q$ given in \cite{mathstack2018_orthoptic} : set any tangency point, and draw the tangent and its reflection with respect to $K$ defining a plane band with width $D(u,\theta)$, then draw point $Q$ on the normal at half distance $D/2$ to tangency point outside caustic.

Finally we have $\theta$ special values, $\pi/3$ and $2\pi/3$, with a family of Poncelet configurations.

\begin{proposition}
For special values, $\theta=\pi/3$ and $\theta=2\pi/3$, the ellipse $\E'$ is inscribed in the $MNL$ triangle.
\end{proposition}

\begin{proof}
Intersect tangents from $M$ to $\E'$ with ellipse $\E$ to get two points $T_1$, $T_2$. Line $T_1T_2$ is tangent to $\E'$ if and only if $\cos^2{\theta}= 1/4$, happening for $\theta=\pi/3$ and $\theta=2\pi/3$.
If $\theta=\pi/3$ circumcircle of $MT_1T_2$ intersects ellipse $\E$ a fourth time at point $M'$ viewing $T_1T_2$ and $\E'$ at $2\pi/3$ angle.

\end{proof}

For these special values, if $\E'$ position is set, we have a Poncelet ellipse-ellipse configuration where $MNL$ is a 3-orbit. Moreover : 

$$k^2 = \frac{\rho^3(\rho+4)^3 1/4}{((\rho+1)^2 + (2\rho-1)1/4)^3}= \frac{16\rho^3(\rho+4)^3}{(4\rho^2+10\rho+3)^3} $$

\section{Fregier circles}
\label{sec:circles}

The Poncelet circle-ellipse(orange color) concentric configuration is an affine transformation of the Poncelet ellipse-ellipse (green and purple colors on \cref{fig:circles}) concentric configuration by scaling distances on the major axis.

\begin{figure}
    \centering
    \includegraphics[width=\textwidth]{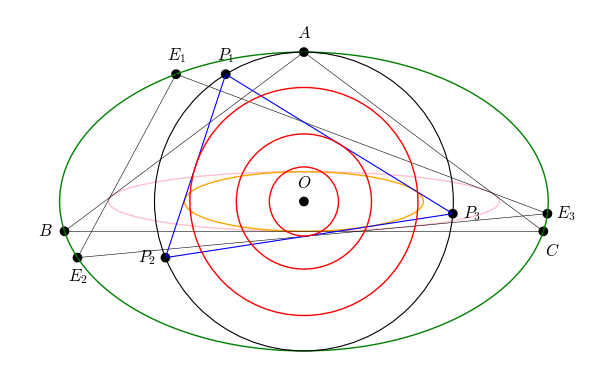}
    \caption{Fregier circles for 3-orbit}
    \label{fig:circles}
\end{figure}

Given a cyclic n-orbit $P_1 \dots P_i \dots P_n$ there are $n$ Frégier concentric circles defined by couples $(P_i,\theta_i)$ where $\theta_i$ is internal angle $P_{i-1}P_iP_{i+1}$.

\begin{proposition}
When $n = 3$, the sum of squared areas of Fregier circles for an $n$- orbit is a Poncelet invariant : $\pi b^2 (1 - \rho/2)$ in the billiard Poncelet configuration with specular reflection on ellipse border.
\end{proposition}

\begin{proof}
Because squared area of the $(P_i,\theta_i)$ Frégier circle is $\cos^2{\theta_i}$ times squared area of outer circle, it suffices to prove that, the sum of squared cosines is the same for any n-orbit. 
Following the idea given in \cite{akopyan2020-invariants} we write the sum as :
$$\sum{\cos^2{\theta_i}} = 1/2\sum{(1 + \cos{(P_{i-1}OP_{i+1})})} = n - 1/(4b^2)\sum{(P_{i-1}P_{i+1})^2} $$ from the circle angle theorem applied to chord $P_{i-1}P_{i+1}$. Equivalently we need to prove that the sum of squared lengths for short diagonals is constant.

When $n=3$ the short diagonals are simply the $P_1P_2P_3$ triangle edges.

Be ABC  the reference triangle for trilinear coordinates.
Cartesian plane axes are defined by major and minor axis, with origin at O.
Projection of $E = E(u)$ on major axis and scaling with ratio $b/a$ along it, is giving cartesian coordinates of $P$ as :

$$x(u) = \frac{ -u(u + 2) s \sqrt{1 - h^2} }{(u^2 + (1 + h)(u + 1))\sqrt{9 - h^2} }  $$
$$y(u) = \frac{ ((h - 1)u^2 + 2(1 + h)u + 2(1 + h))s \sqrt{1 - h}}{ (((u^2 + (1 + h)(u + 1))(3 - h)\sqrt{3 + h} }$$

Five tangent lines from vertices $A,B,C$ to caustic ($ABC$ Mandart inellipse) are used to get its conic functions :
$$ (h + 1)^4 : (1 - h)^2 : (1 - h)^2 : -2(1 - h)^2 : -2(1 + h)^2(1 - h) : -2(1 + h)^2(1 - h) $$
and by duality, tangent lines coordinates $\alpha : \beta : \gamma$  are verifying equation 
$$(1 + h)^2\beta\gamma + (1 - h)\gamma\alpha + (1 - h)\alpha\beta = 0 $$

Be $M(X)$ a generic point different from $E_1(u)$ with $u=u_1$, then line $E_1M =  u X : 1 : (u + 1)(X + 1) $ belongs to caustic tangents if its trilinear coordinates are satisfying the caustic dual equation.
So $X$ is satisfying binomial equation :  $$X^2 - \mu X + \psi = 0 $$
with
$$\mu = \frac{(1 - h)u^2 + (3 + h^2)u + (1 + h)^2 }{(h - 1)(u + 1)u}$$
$$\psi = \frac{-(h + 1)^2}{(h - 1)u)} $$
which has roots $u_2,u_3$ ($u_2 + u_3 = \mu$; $u_2u_3 = \psi$) defining the two tangents $E_1E_2$ and $E_1E_3$ from $E$ to the caustic.

It results from CAS computation that the sum of the 3-orbit  squared edges lengths is $$\sum_{cyc}{(x(u_i+1) - x(u_i))^2 + (y(u_i+1) - y(u_i))^2 } = 2b^2(\rho + 4)$$
and the sum of areas of Fregier circles has constant value
$\pi b^2 (1 - \rho/2)$ .
\end{proof}

\begin{conjecture}
Previous proposition 
is valid for every integer $n$ greater than 3.
\end{conjecture}

\textbf{Dynamical Systems remark} :  the Fregier circles are an example of generalized Poncelet configuration, where the n-concentric inner circles are in the same pencil defined by two of them. For $n=3$, if $R$ is radius of outer circle and $r_1,r_2,r_3$ are radi of inner concentric circles, the proposition has a corollary for a 3-orbit $P_1P_2P_3$ existence condition (equivalent to $\rho$ defined between $0$ and $1/2$) :
$$ 3/4 \leq  (r_1^2 + r_2^2 + r_3^2)/R^2 \leq 1 $$.

\section{Videos and Symbols}
\label{sec:videos}
Animations illustrating some phenomena herein are listed on Table~\ref{tab:playlist}.

\begin{table}
\small
\begin{tabular}{|c|l|l|}
\hline
id & Title & \textbf{youtu.be/<.>}\\
\hline
01 & {Frégier phenomena i: Area-invariant envelope of chords.} &
\href{https://youtu.be/
UCCG5AT8dh8}{\texttt{UCCG5AT8dh8}}\\
02 & {Frégier phenomena iii: Circular envelopes and a new invariant } &
\href{https://youtu.be/AzNXeBU2NTI}{\texttt{AzNXeBU2NTI}}\\

\hline
\end{tabular}
\caption{Videos of some focus-inversive phenomena. The last column is clickable and provides the YouTube code.}
\label{tab:playlist}
\end{table}

\cref{tab:symbols} provides a quick-reference to the symbols used in this article.

\begin{table}
\begin{tabular}{|c|l|}
\hline
symbol & meaning \\
\hline
$\E,\E'$ & outer and inner ellipses \\
$O$ & center of $\E,\E'$ \\
$a,b$ & outer ellipse semi-axes' lengths \\
$\rho$ & ratio of inradius to circumradius $r/R$ \\
$x:y$ & cartesian coordinates \\
$\alpha:\beta:\gamma$ & trilinear coordinates \\
$u,u_i$ & parameter of point on $\E$ \\
$h:s$ & elliptic billiard parameters \\
$E_i$ & point on $\E$ \\
$P_i$ & point on outer circle \\
$r_i$ & Frégier circle radius \\
\hline
\end{tabular}
\caption{Symbols used in the article.}
\label{tab:symbols}
\end{table}

\section{Acknowledgments}
\label{sec:ack}
\noindent We would like to thank D. Reznik for observations and his youtube videos about it. 

\appendix

\bibliographystyle{maa}
\bibliography{999_refs,999_refs_rgk}

\end{document}